\documentclass[letterpaper, 10 pt, conference]{ieeeconf}

\IEEEoverridecommandlockouts 

\usepackage{graphics} % for pdf, bitmapped graphics files
\usepackage{epsfig} % for postscript graphics files
\usepackage{mathptmx} % assumes new font selection scheme installed
\usepackage{times} % assumes new font selection scheme installed
\usepackage{amsmath} % assumes amsmath package installed
\usepackage{amssymb}  % assumes amsmath package installed

\usepackage{color}
\usepackage{epsfig}
\usepackage{cases}
\usepackage{epstopdf}
\usepackage{cases}
\usepackage[colorlinks,
linkcolor=blue,
anchorcolor=green,
citecolor=blue
]{hyperref}
\usepackage{cite}
\newtheorem{theorem}{Theorem} 

\newtheorem{lemma}{Lemma}

\newtheorem{definition}{Definition}

\newtheorem{remark}{Remark}

\newtheorem{assumption}{Assumption}

\newcommand{\nm}{\nonumber}
\newcommand{\bR} {{\mathbb R}}

\newcommand{\1}{\mbox{1}\hspace{-0.25em}\mbox{l}}
\newcommand{\0}{\mbox{0}\hspace{-0.4em}\mbox{0}}

\title{\LARGE \bf
Distributed Formation Control of Multi-Robot Systems: 
\\A Fixed-Time Behavioral Approach}

\author{Ning Zhou, Xiaodong Cheng, Yuanqing~Xia, Yanjun Liu% <-this % stops a space
\thanks{The work was supported in part by the National Natural
Science Foundation of China under Grant 61603095, Grant 61972093, Grant 61720106010, and Grant 61973147. The work of Yuanqing Xia was supported in part by the Science and Technology on Space Intelligent Control Laboratory under Grant KGJZDSYS-2018-05.}% <-this % stops a space
\thanks{Ning Zhou is with School of Electrical Engineering, Hebei University of Science and Technology, Shijiazhuang 050018, China. {\tt\small  zhouning2010@gmail.com}.}
\thanks{Xiaodong Cheng is with Department of Electrical Engineering, Eindhoven University of Technology, 5600 MB Eindhoven, The Netherlands. {\tt\small x.cheng@tue.nl}.}
\thanks{Yuanqing Xia is with the School of Automation,
Beijing Institute of Technology, Beijing 100081, China.
{\tt\small xia\_yuanqing@bit.edu.cn}.}
\thanks{Yanjun Liu is with the College of Science, Liaoning University of Technology,
Jinzhou 121001, China. {\tt\small liuyanjun@live.com}.}}

\begin{document}
\maketitle
\thispagestyle{empty}
\pagestyle{empty}
\begin{abstract}
This paper investigates a distributed formation control problem for networked robots, with the global objective of achieving predefined time-varying formations in an {environment with obstacles}.  A novel fixed-time behavioral approach is proposed to tackle the problem, where a global formation task is divided into two local prioritized subtasks, and each of them leads to a desired velocity that can achieve the individual task in a fixed time. Then, two desired velocities are combined 
via the framework of the null-space-based behavioral projection, leading to a desired merged velocity that guarantees the fixed-time
convergence of task errors. Finally, the effectiveness of the proposed control method is demonstrated by simulation results.

\end{abstract}
%%%%%%%%%%%%%%%%%%%%%%%%%%%%%%%%%%%%%%%%%%%%%%%%%%%%%%%%%%%%%%%%%%%%%%%%%%%%%%%%
\section{Introduction \label{sec1}}

Due to broad applications of multi-robot systems in e.g., search and rescue missions, natural resource monitoring, and outdoor industrial operations such as fault diagnosis and repair, control of multi-robot systems has attracted increasing attention from the systems and control domain \cite{zou2019event,Zhou2020CYB}. However, in many practical applications, autonomous robots are deployed in a dynamical environment to perform multiple parallel tasks, for example, to maintain the desired formation and avoid moving obstacles at the same time. An efficient and reliable control scheme for operating this kind of systems poses a challenge in our domain.

To achieve multi-mission control problems, the so-called behavioral approach is resorted \cite{antonelli2006kinematic,antonelli2010flocking,ott2015prioritized}. In this approach, a comprehensive
control task is decomposed into multiple smaller
and simpler subtasks, characterized by \textit{behavioral functions}, and each of them provides a motion command. Merging these commands through a certain method then leads to the eventual control law. One of the commonly used merging methods is the \textit{Null-Space-Based Behavioral} (NSB) approach  \cite{balch1998behavior, ott2015prioritized}.
This approach is used to handle the situation when some tasks, or behaviors, {conflict with} each other, e.g. a team of robots are supposed to form certain formation when they also have to avoid obstacles appearing on their way. With the NSB approach, the behaviors are prioritized, where the lower prioritized  tasks are projected to the null space
of higher-priority tasks, guaranteeing that they do
not contradict the higher ones. The NSB allows the entire networked systems to exhibit the robustness with respect to eventually conflicting tasks/behaviors \cite{zhou2018neural}, and it shows great potential in real-world applications \cite{antonelli2010flocking,huang2019adaptive,baizid2017behavioral}. Recently, the behavioral approach has been extended to a decentralized/distributed manner \cite{antonelli2010flocking,ahmad2014behavioral}, in which controllers are localized in each autonomous robot and can only acquire information from its neighboring robots. 
A decentralized framework of behavioral approaches was firstly given in \cite{antonelli2010flocking}, but the theoretical guarantee of the convergence of behavior errors is lacking. A distributed formation control method for multi-robot systems using NSB is provided in \cite{ahmad2014behavioral}, which results in the asymptotic stability of the closed-loop system. However, this method only focuses on a triangular formation control problem.  Inspired by the works in \cite{antonelli2010flocking,ahmad2014behavioral}, we investigate a distributed {NSB} control scheme for formation control of multi-robot systems. {Differently}, we consider obstacles in the environment and use a distributed estimator \cite{zhou2017finite} to enable a distributed control scheme.

The major difference between the current work and the existing approaches is that we propose a fixed-time behavioral approach, {which provides faster convergence speed and better control precision, in comparison to the asymptotic results \cite{zhou2017finite,ahmad2014behavioral,huang2019adaptive}, and initial-state-independent convergence time compared with the finite-time approaches \cite{zhou2017finite,huang2019adaptive}}. This is motivated by the need for multi-robot applications requiring a fast convergence speed and a high control accuracy, e.g., cooperative robotic imaging. Note that the fixed-time control, which is firstly presented in \cite{polyakov2012nonlinear}, in general, has a fast convergence rate, high-precision control performance, and disturbance rejection properties \cite{zuo2018overview}. In this paper, we will combine the behavioral approach and fixed-time control to achieve multiple tasks of networked robots in a fixed time. To the best of our knowledge, such a problem has not been addressed in the literature so far. 
To solve this problem, we introduce the behavior functions of collision avoidance and cooperative formation, respectively, which lead to two velocity commands using the fixed-time design and inverse kinematics method. These two commands are merged in priority, via the null-space-based behavioral projection, to give a desired velocity for each agent that can be computed based on only local information.  Using a distributed fixed-time estimator, a  fixed-time behavioral approach is designed and implemented in a distributed manner. The theoretical proof of the fixed-time convergence of task errors is provided. 
The developed fixed-time behavioral approach can handle time-varying formations in a distributed framework and guarantee collision/obstacle avoidance, and it is not only limited to some certain triangle-based formations in an ideal environment.

The rest of the paper is organized as follows: Section~\ref{sec2} recaps the concept of fixed-time control and formulates the problem; Section \ref{sec3} presents the main result of this paper, which provides a fixed-time behavioral control scheme for multi-robot systems; The simulation result is provided in Section \ref{sec5}, and concluding remarks are made in Section~\ref{sec6}.

\textit{Notation:} The set of real numbers is denoted by $\bR$. For a vector or matrix, $\| \cdot \|$ denotes its Euclidean norm. The $i$-th element of a vector $v$ is denoted by $v_{i}$.
The $n$-dimensional vector whose elements are all $1$ is denoted by $\1_{n} \in \mathbb R^n$. The involution operation without loss of the number's sign is represented by $ x^{[p]}:=|x|^p\mathrm{sgn}(x), x, p\in \bR $. $\mathrm{sgn}(\cdot )$ is the sign function that returns $-1$, $0$ or $1$.

\section{Preliminaries and Problem Setting \label{sec2}}

\subsection{Fixed-Time Stability}

Consider a nonlinear system
\begin{align}\label{system0}
\dot{x}(t)=f(t,x),\ x(0)=x_{0},
\end{align}
with $ x(t) \in \bR^{n} $ and the nonlinear function $f(t,x)$. If $ f(t,x) $ is discontinuous, the solutions of \eqref{system0} are \textit{Filippov}. Suppose the origin is an equilibrium point of \eqref{system0}, then the fixed-time stability is defined as follows.

\begin{definition} \cite{polyakov2012nonlinear}
The origin $x = 0$ is said to be \textit{globally fixed-time stable} if it is globally asymptotically stable and any solution $x(t,x_0)$ of \eqref{system0} reaches $x = 0$ in some settling time $t = T(x_0)$ and remains there for all $t \geq T(x_0)$, where $T(x_0)$ is globally bounded by some number $T_{\max} \in \mathbb{R}_{>0}$. 
\end{definition}

Notice that in the concept of the fixed-time stability, the settling (convergence) time $T(x_0)$ is always bounded independent of the initial condition $x_0$. Furthermore, the fixed-time stability of the nonlinear system \eqref{system0} can be characterized by the following lemma.
\begin{lemma}\cite{polyakov2012nonlinear,parsegov2013fixed}
\label{ld1} 
If there exists a continuous radially unbounded and positive definite function $ V:\bR^{n}\to\bR_{>0}$ such that $ V(x)=0$ if and only if $x=0 $, and any solution $ x(t,x_0) $ of \eqref{system0} satisfies
\begin{align}\label{V0}
\dot{V}(x)&\le -\eta_{1}V^{k_{1}}(x)-\eta_{2}V^{k_{2}}(x),\\
{\rm or}\ \dot{V}(x)&\le -(\eta_{1}V^{k_{3}}(x)+\eta_{2}V^{k_{4}}(x))^{k_{5}},\label{V01}
\end{align}
where $ \eta_{1},\eta_{2},k_{1},k_{2},k_{3},k_{4},k_{5} \in \mathbb{R}_{>0}$ with $ k_{1}>1 $, $ 0<k_{2}<1 $, $ k_{3}k_{5}>1 $, and $ k_{4}k_{5}<1 $, then the origin of \eqref{system0} is globally fixed-time stable and the settling time function $ T $ can be estimated by 
\begin{align*}
&T\le T_{\max}:=\frac{1}{\eta_{1}(k_{1}-1)}+\frac{1}{\eta_{2}(1-k_{2})}, \\
\text{or} \ & T\le T_{\max}:= \frac{1}{\eta_{1}^{k_{5}}(k_{3}k_{5}-1)}+\frac{1}{\eta_{2}^{k_{5}}(1-k_{4}k_{5})},
\end{align*}
where $T_{\max}$ is independent on the initial condition $x(0)$. 
\end{lemma}

\subsection{Problem Setting}\label{sec:setting}

Consider a network of $ n $ mobile robots, and each of them has the following dynamics:
\begin{align}
\dot{p}_{i}(t) &= v_{i}(t), \
p_{i}(t)=[x_{i}(t),y_{i}(t),z_{i}(t)]^{\top},\label{p}
\end{align}
where $i = 1, 2, \cdots, n$. The positions and velocities of the overall system are represented by two stacked vectors 
\begin{align} \label{eq:pv}
p(t)=[p_{1}(t),\dots,p_{n}(t)]^{\top},\ v(t)=[v_{1}(t),\dots,v_{n}(t)]^{\top}.
\end{align}

The robots are interconnected via a communication network, whose topology is an undirected graph $G=(V, \mathcal{E})$,
with $V = \{1, 2, \cdots, n\}$ the set of nodes and $\mathcal{E} \subseteq V \times V$
the set of edges. An edge $(i, j) \in \mathcal{E}$  if and only if there exists an information exchange between robots $i$ and $j$, and
the weight of $(i, j)$, denoted by $a_{ij} \in\bR_{\ge0}$, represents the communication strength. The Laplacian matrix $L$ of $G$ is thereby defined as $[L]_{ij} = -a_{ij}$ if $i \ne j$, and $[L]_{ii} = \sum_{j=1}^{n}a_{ij}$ otherwise.

In this paper, we resort to the virtual leader approach \cite{porfiri2007tracking} to achieve the time-varying formation of the robots. In this scheme, a virtual leader is specified as
a time-varying reference and the robots are designed to maintain desired offsets with respect to the position of the
virtual leader $p_0(t)$. 
Define $ B={\rm{diag}}\{b_{1},\dots,b_{n}\}\in\bR^{n\times n} $, where $ b_{i} > 0$ if the information of the virtual leader is available to the follower robot $i$, and $ b_{i} = 0 $ otherwise. Then, we let   $ H:=L+B$. Assuming that $G$ is {strongly} connected, and at least one robot can acquire information from the leader, then we obtain that $H$ is positive definite.

Given a desired relative position $ p_{i0} $ between the robot $ i $ and the leader and $ d \in \mathbb{R}_{>0}$ the radius of circular repulsive
zone of each robot, the control objective in this paper is to design a distributed fixed-time controller for each robot $i$ such that both targets $ \|p_{i}(t)-p_{0}(t)-p_{i0}\|=0$ and $ \|p_{i}(t)-p_{i}^{o}(t)\|\ge d $ are achieved for all $ t\ge T_{i} $, where $ T_{i}>0 $ is a prefixed time which can be designed, and $ p_{i}^{o} $ denotes the position of the nearest obstacle of the robot $ i $.

\section{Fixed-Time Behavioral Control}\label{sec3}
Two types of behaviors are considered in this paper, namely,
the collision avoidance behavior and cooperative formation behavior,
which yield two desired velocities. Both velocities guarantee
fixed-time convergence, and then they are prioritized and
combined as a merged velocity that guarantees the fixed-time convergence of the global task.

\subsection{Collision Avoidance Behavior}\label{sec3.1}
The primary concern for a team of robots to perform tasks in an environment with obstacles is to avoid collision with these obstacles as well with the
other robots. This is referred to as collision avoidance behavior.  
We define $\rho_{io}:\bR^{3}\to\bR_{>0}$ as a behavior function for every individual robot:
\begin{align}
\rho_{io} :={\frac{1}{2}\|p_{i}-p_{i}^{o}\|^{2}}, \label{1}
\end{align} 
with $p_{i}^o:\bR_{>0}\to\bR^{3}$ the position of the nearest obstacle (or teammate) of the robot $i$. Then the behavior-dependent Jacobian matrixes are given as 
\begin{align}
\label{jio}
J_{io} = \frac{\partial \rho_{io}}{\partial p_{i}} ={(p_{i}-p_{i}^o)^{\top}},\ J_{i}^{o}=\frac{\partial \rho_{io}}{\partial p_{i}^{o}} =-{(p_{i}-p_{i}^o)^{\top}},
\end{align}
with the right pseudoinverse $ J_{io}^{\dag}=({p_{i}-p_{i}^o})/{\|p_{i}-p_{i}^o\|^{2}} $. We define a circular \textit{repulsive zone} around each obstacle/robot, with the coordinates of the object as the center and $d\in\bR_{>0}$ as the radius. Then, we define the CAB task error as 
\begin{equation}\label{2}
\tilde{\rho}_{io} := \rho_{od}-\rho_{io}, \ \text{with}~\rho_{od}: = \frac{d^{2}}{2},
\end{equation}%
from which the desired collision-avoidance behavioral velocity $\dot{x}_{io} :\bR_{\ge0}\to\bR^{3} $ is designed for the agent $ i $ as 
\begin{align}\label{v1}
v_{io} =& J_{io}^{\dag} \left[\lambda_{io}(\beta_{1}\tilde{\rho}_{io}^{[r_{1}]}+\beta_{2}\tilde{\rho}_{io}^{[r_{2}]})^{[r_{0}]}-J_{i}^{o}v_{i}^{o}\right],
\end{align}
where $ \beta_{1},\beta_{2},r_{0},r_{1},r_{2} \in \bR_{>0} $ are design parameters with $r_{1}r_{0}>1$, $r_{2}r_{0}<1$, $r_{0}\in(\frac{1}{2},1) $, and $ v_{i}^{o} $ is velocity of the nearest obstacle/teammate of robot $ i $.

The following lemma shows that the desired velocity in
\eqref{v1} guarantees collision avoidance for each agent. 

\begin{lemma}\label{l1}
Consider the collision avoidance behavior function \eqref{1}. Suppose that each agent $i \in V$ is driven
by the desired velocity \eqref{v1}. If $ \|p_{i}(0)-p_{i}^{o}(0)\|\le d $, then for any $\tilde{\rho}_{io}(0)\in\bR$, there exists a settling time $T_{i,o}>0$ such that $ \|p_{i}(t)-p_{i}^{o}(t)\|\ge d $,  $\forall~t\geq T_{i,o}$.
\end{lemma}
\begin{proof}
Consider a Lyapunov function as follows 
$
V_{i,o}=\frac{1}{2}\gamma_{i,o}\tilde{\rho}_{io}^{2},
$ 
where $\gamma_{i,o}>0$ is to be determined. Differentiating
$V_{i,o}$  yields
\begin{align*}
\nm\dot{V}_{i,o}=&-\gamma_{i,o}\tilde{\rho}_{io}(J_{io}v_{io}+J_{i}^{o}v_{i}^{o})
\\ \nm=&-\gamma_{i,o}\tilde{\rho}_{io}
J_{io}J_{io}^{\dag}[\lambda_{io}(\beta_{1}\tilde{\rho}_{io}^{[r_{1}]}+\beta_{2}\tilde{\rho}_{io}^{[r_{2}]})^{[r_{0}]}-J_{i}^{o}v_{i}^{o}]\\&-\gamma_{i,o}\tilde{\rho}_{io}J_{i}^{o}v_{i}^{o}
\\
=&-\left[ \gamma^{i}_{o1}V_{i,o}^{\frac{r_{1}r_{0}+1}{2r_{0}}}+\gamma^{i}_{o2}V_{i,o}^{\frac{r_{2}r_{0}+1}{2r_{0}}}\right]^{r_{0}},
\end{align*}
with the two positive scalars  $ \gamma^{i}_{o1} = (\gamma_{i,o}\lambda_{i,o})^{\frac{1}{r_{0}}}(\frac{2}{\gamma_{i,o}})^{\frac{r_{1}r_{0}+1}{2r_{0}}}\beta_{1}$, $ \gamma^{i}_{o2} =(\gamma_{i,o}\lambda_{i,o})^{\frac{1}{r_{0}}}(\frac{2}{\gamma_{i,o}})^{\frac{r_{2}r_{0}+1}{2r_{0}}}\beta_{2} $.
From Lemma~\ref{ld1}, we have $\tilde{\rho}_{io}$
converges to 0 in a fixed time
\begin{equation*}
T_{
i,o}\leq\frac{1}{(\gamma^{i}_{o1})^{r_{0}}(r_{1}r_{0}-1)}+\frac{1}{(\gamma^{i}_{o2})^{r_{0}}(1-r_{2}r_{0})},
\end{equation*}  
irrespective of the initial conditon $\tilde{\rho}_{io}(0)$.
\end{proof}
\subsection{Fixed-Time State Estimator and Cooperative Behavior}\label{sec3.2}
This section considers the cooperative formation behavior, where all
the robots are moving towards a predefined formation following to a virtual leader. More specifically, we aim to achieve, for each robot $i$, 
$ \|p_{i}(t)-p_{0}(t)-p_{i0}\|=0$ for all $ t\ge T_{i} $, with $T_i$ a fixed settling time.

To realize our scheme in a distributed manner, we first design a fixed-time sliding mode estimator $\hat{p}_{i}$ for each robot~$ i $ to estimate the leader's information from its neighbors.
\begin{align}
\dot{\hat{p}}_{i}(t)=& -K_{1}\eta_{i}(t)-K_{2}\eta_{i2}(t)-K_{3}{\rm{sgn}}(\eta_{i3}(t)),\label{11}\\
{\eta}_{i}(t)=& \sum_{j=0}^{n}a_{ij}(\bar{p}_{i}(t)-\bar{p}_{j}(t))^{[\frac{r_{3}}{r_{4}}]},\ a_{i0}=b_{i},\nm\\
{\eta}_{i2}(t)=& \sum_{j=0}^{n}a_{ij}(\bar{p}_{i}(t)-\bar{p}_{j}(t))^{[\frac{r_{5}}{r_{6}}]},\nm\\
{\eta}_{i3}(t)=& {\rm sgn} (\sum_{j=0}^{n}a_{ij}(\bar{p}_{i}(t)-\bar{p}_{j}(t))),\nm
\end{align}
where $K_{1},K_{2},K_3\in\bR_{>0}$ are gains to be designed, {$ \bar{p}_{i}(t)=\hat{p}_{i}(t)-p_{0}(t) $, $ \bar{p}_{j}(t)=\hat{p}_{j}(t)-p_{0}(t) $, $\bar{p}(t) =[\bar{p}_{11}^{\top}(t),\dots,\bar{p}_{n1}^{\top}(t)]^{\top} $,} $r_{3},r_{4},r_{5},r_{6}\in\bR_{>0} $ are design parameters with {$ \frac{r_{3}}{r_{4}}>1 $ and $\frac{r_{5}}{r_{6}}<1 $}. The virtual leader $p_0$, $a_{ij}$, and $b_{i}$ are defined in Section~\ref{sec:setting}.  
\begin{assumption}\label{a3.1}
Suppose that the supremum of the leader's velocity is bounded as $  K_{3}\ge \sup_{t\ge 0}\|\dot{p}_{0}(t)\|_{\infty} $, where $ K_3\in\bR_{>0} $ is the gain of estimator \eqref{11}.
\end{assumption}
\begin{lemma}\label{l8}
If Assumption \ref{a3.1} is satisfied, then the estimator \eqref{11} converges to the leader's state $ p_{0}(t) $ in a fixed time, i.e., for any $ (\hat{p}_{i}(0),p_{0}(0))\in \bR^{3} \times \bR^{3}  $, there exists $ T_{E}>0 $ such that $\hat{p}_{i}(t)\equiv p_{0}(t) $, $\forall~t\ge T_{E} $.
\end{lemma}
\begin{proof}
Consider a Lyapunov candidate
\begin{align*} 
V_{E}(\bar{p}):=\frac{1}{2}\bar{p}^{\top}(H\otimes I_{3})\bar{p},
\end{align*}
where $\bar{p} (t) =[\bar{p}^{\top}(t),\dots,\bar{p}_{n}^{\top}(t)]^{\top} $, $ \bar{p}_{i}(t)=\hat{p}_{i}(t)-p_{0}(t) $. The time derivative of $ V_{E} $ is computed as
\begin{align}\label{13}
\dot{V}_{E}&=\bar{p}^{\top}(H\otimes I_{3})(
[\dot{\hat{p}}_{1}^{\top}(t),\dots,\dot{\hat{p}}_{n}^{\top}(t)]^{\top}
-\1_{n}\otimes \dot{p}_{0})\nm\\
=& \bar{p}^{\top}(H\otimes I_{3})[-K_{1}(H\bar{p})^{[\frac{r_{3}}{r_{4}}]}-K_{2}(H\bar{p})^{[\frac{r_{5}}{r_{6}}]}-K_{3}{\rm{sgn}}(H\bar{p})\nm\\&  -\1_{n}\otimes \dot{p}_{0}]\nm\\
\le& -K_{1}\bar{p}^{\top}(H\otimes I_{3})((H\otimes I_{3})\bar{p})^{[\frac{r_{3}}{r_{4}}]}-K_{2}\bar{p}^{\top}(H\otimes I_{3})((H\otimes I_{3}) \nm\\& \times\bar{p})^{[\frac{r_{5}}{r_{6}}]} -{(K_{3}-{\sup_{t\ge 0}\| \dot{p}_{0}(t)\|_{\infty}})}\|-K_{2}\bar{p}^{\top}(H\otimes I_{3})\bar{p}\|_{1}\nm\\
\le&- \kappa_1{V}_{E}^{\iota_1}- \kappa_2{V}_{E}^{\iota_2},
\end{align}
where $  \kappa_1:= K_{1}(\frac{2\lambda_{\max}^{2}(H)}{\lambda_{\min}(H)})^{\frac{2r_{4}}{r_{3}+r_{4}}}$, $ \kappa_2:= K_{2}(\frac{2\lambda_{\max}^{2}(H)}{\lambda_{\min}(H)})^{\frac{2r_{6}}{r_{5}+r_{6}}}$, $ \iota_1:= \frac{r_{3}+r_{4}}{2r_{4}}$, $ \iota_2:=\frac{r_{5}+r_{6}}{2r_{6}} ${, and $ \lambda_{\min}(H) $ and $ \lambda_{\max}(H) $ denote the minimum and maxmum eigenvalues of $ H $, respectively}.
It follows from Lemma \ref{ld1} that for all $ (\hat{p}_{i}(0),p_{0}(0))\in \bR^{3} \times \bR^{3}  $, there exists $ T_{E}:=\frac{1}{\kappa_1(\iota_1-1)}+\frac{1}{\kappa_2(1-\iota_2)} $ such that $\hat{p}_{i}(t)\equiv p_{0}(t) $ for any $ t\ge T_{E} $.
\end{proof}

Using the distributed estimator in \eqref{11}, we define the cooperative behavior function of the robot $i$ as 
\begin{align}
\label{ctbf}
\rho_{if}&:= \frac{1}{2}\|p_{i}-\hat{p}_{i}-p_{i0}\|^{2}. 
\end{align}
with a desired relative position $p_{i0}$ from the robot $i$ to the virtual leader.
The task error is defined as  
\begin{align}\label{15}
{\tilde{\rho}_{f}:=[\tilde{\rho}_{1f},\dots,\tilde{\rho}_{nf}]^{\top}, \ \text{with} \ \tilde{\rho}_{if}:= -{\rho}_{if}}
\end{align}
The Jacobian matrices related to \eqref{ctbf} are defined as 
\begin{align*} 
J_{\hat{f}}=&{\rm blkdiag}\left\{\frac{\partial \rho_{1f}}{\partial \hat{p}_{1}},\dots,\frac{\partial \rho_{nf}}{\partial \hat{p}_{n}}\right\} \in \mathbb{R}^{n\times
3n},
\\ \nm
J_{0}=&{\rm blkdiag}\left\{\frac{\partial \rho_{1f}}{\partial {p}_{10}},\dots,\frac{\partial \rho_{nf}}{\partial {p}_{n0}}\right\} \in \mathbb{R}^{n\times
3n},
\\
{J_{f}}=&{\rm blkdiag}\left\{\frac{\partial \rho_{1f}}{\partial p_{1}},\dots,\frac{\partial \rho_{nf}}{\partial p_{n}}\right\} \in \mathbb{R}^{n\times
3n}, \ \text{with} 
\\
{J_{f}^{\dagger}}=&
\mathrm{blkdiag}\left\{\frac{(p_{1}-\hat{p}_{1}-p_{10})}{\|p_{1}-\hat{p}_{1}-p_{10}\|^{2}} \right., \left.\dots,\frac{(p_{n}-\hat{p}_{n}-p_{n0})}{\|p_{n}-\hat{p}_{n}-p_{n0}\|^{2}}\right\}
\end{align*}
where $\frac{\partial \rho_{if}}{\partial p_{i}} = (p_{i}-\hat{p}_{i}-p_{i0})^{\top}$, $\frac{\partial \rho_{if}}{\partial \hat{p}_{i}} = \frac{\partial \rho_{if}}{\partial {p}_{i0}} = -(p_{i}-\hat{p}_{i}-p_{10})^{\top}$. Based on these,
we design the desired cooperative behavior velocity $ v_{f} $ as follows:
\begin{align}\label{v2}
v_{f}=&[v_{1f}^{\top},\dots,v_{nf}^{\top}]^{\top}, \ \text{with} 
\\
v_{if}:=&J^{\dagger}_{if} [\lambda_{f}(\beta_{1}\tilde{\rho}_{if}^{[r_{1}]}+\beta_{2}\tilde{\rho}_{if}^{[r_{2}]})^{[r_{0}]}-J_{i\hat{f}}\dot{\hat{p}}_{i}-J_{i0}\dot{p}_{i0}].\nm
\end{align}
The gain $\lambda_{f} \in \mathbb{R}_{>0}$, and $ \beta_{1},\beta_{2}, r_{0},r_{1},r_{2} \in \bR_{>0} $ are design parameters satisfying $r_{1}r_{0}>1$, $r_{2}r_{0}<1$, $r_{0}\in(\frac{1}{2},1) $. Note that $v_{f}$ can be presented in a compact form as
\begin{equation} \label{v22}
v_{f} = J^{\dagger}_{f} [\Lambda_{f}(\beta_{1}\tilde{\rho}_{f}^{[r_{1}]}+\beta_{2}\tilde{\rho}_{f}^{[r_{2}]})^{[r_{0}]}-J_{\hat{f}}\dot{\hat{p}}-J_{0}v_{*0}],
\end{equation}
where $\Lambda_{f}=\lambda_{f}I\in \mathbb{R}^{3n \times 3n}$, $\dot{\hat{p}} =[\dot{\hat{p}}_{1}^{\top},\dots,\dot{\hat{p}}_{n}^{\top}]^{\top} $, and $ v_{*0}:=[\dot{p}_{10},\dots,\dot{p}_{n0}]^{\top} $. With the distributed estimator \eqref{11} and the desired velocity in \eqref{v22}, we have the following lemma.
\begin{lemma}\label{l4}
Consider the cooperative behavior error $\tilde{\rho}_{f}$ in \eqref{15}. If each robot $i \in V$ is driven by the desired velocity in \eqref{v2}, then, for any initial condition $\tilde{\rho}_{f}(0)$, there exists a settling time $T_{f}>0$ such that $ \|p_{i}-\hat{p}_{i}-p_{i0}\|=0 $, $ \forall~t\geq T_{f}$.
\end{lemma}
\begin{proof}
This lemma can be proved using the similar procedure as the proof of Lemma~\ref{l1}. However, we use a different Lyapunov function  
$ V_{f}:=\frac{1}{2}\gamma_{f}\tilde{\rho}_{f}^{\top}\tilde{\rho}_{f},$
where $\gamma_{f}>0$ is the design parameter in \eqref{v2}. Taking the derivative of
$V_{f}$ with respect to time then leads to
\begin{align*}
\nm\dot{V}_{f}=&-\gamma_{f}\tilde{\rho}_{f}^{\top}\dot{\tilde{\rho}}_{f}  = -\gamma_{f}\tilde{\rho}_{f}^{\top}(J_{f}v_{f}+J_{\hat{f}}\dot{\hat{p}}_{1}+J_{0}v_{*0})
\\ \nm=&-\gamma_{f}\tilde{\rho}_{f}^{\top}
J_{f}J^{\dagger}_{f} [(\beta_{1}\tilde{\rho}_{f}^{[r_{1}]}+\beta_{2}\tilde{\rho}_{f}^{[r_{2}]})^{[r_{0}]}-J_{\hat{f}}\dot{\hat{p}}-J_{0}v_{*0}]\\&-\gamma_{f}\tilde{\rho}_{f}^{\top}J_{\hat{f}}\dot{\hat{p}}-\gamma_{f}\tilde{\rho}_{f}^{\top}J_{0}v_{*0}\\
=&-\gamma_{f}\lambda_{f}\tilde{\rho}_{f}^{\top}(\beta_{1}\tilde{\rho}_{f}^{[r_{1}]}+\beta_{2}\tilde{\rho}_{f}^{[r_{2}]})^{[r_{0}]}.
\end{align*}
The rest of this proof follows similarly as the proof of Lemma~\ref{l1}, and details are omitted to conserve space.
\end{proof}
\subsection{Merged Desired Velocities}
In this section, we design a desired velocity by merging the two behaviors in the above two sections. This merging is taken
using the null-space-based behavioral approach, where the collision avoidance behavior is given a higher priority. Specifically, we determine
the desired velocity as follows:
\begin{align}\label{v}
v_{id}=v_{io}+(I-J_{io}^{\dag}J_{io})v_{if},~\forall~i \in V.
\end{align}
where $v_{io}$ and $v_{if}$ are given in \eqref{v1} and \eqref{v2}.
With this merged velocity, we prove that each robot can achieve both tasks simultaneously within a fixed settling time.
\begin{theorem}\label{taskl4}
Consider the collision avoidance behavior in \eqref{1} and the cooperative behavior in \eqref{ctbf}. If each robot $i \in V$ is driven by the merged desired velocity in \eqref{v}, then for any $(\tilde{\rho}_{io}(0),\tilde{\rho}_{if}(0))$, there exists a settling time $ T_{i}$ such that 
$
\|p_{i}-p_{i}^{o}\|\ge d, \ \text{and} \ \|p_{i}-\hat{p}_{i}-p_{i0}\|=0, \
\forall~t\geq {T_{i}}
$. 
\end{theorem}

Note that the proof of the theorem is not just a simple combination of the conclusions of Lemma~\ref{l1} and Lemma~\ref{l4}. When there is no conflict between the two tasks, we have $J_f J_{io}^{\dag} = 0$, which means that two tasks in the velocity space are orthogonal and thus the fixed-time properties can be proved independently. However, if $J_f J_{io}^{\dag} \ne 0$, i.e., the tasks are conflicting, then the proof becomes nontrivial. We present the detailed proof in the Appendix.
{\begin{remark}
\label{rem:parameter}
To adjust the convergence time of tracking errors, we can tune the parameters in \eqref{v} according to the formula of settling time $ T_{i} $. For example, the larger values of $\beta _1$ and $\beta _2$, or the smaller values of $ \gamma_{i,o} $ and $ \gamma_{f} $, will lead to a faster convergence speed.
\end{remark}}
\begin{figure}
\begin{center}
\includegraphics[width=0.43\textwidth]{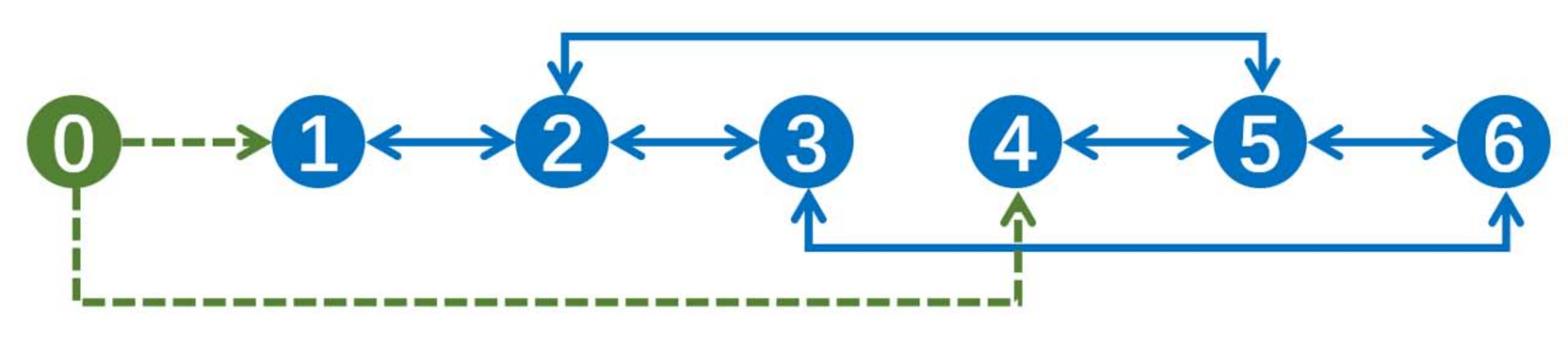}
\caption{The autonomous robots network.} 
\label{fig:1}
\end{center}
\end{figure}
\section{Simulation Results}\label{sec5}
Consider a multi-agent systems connected by an undirected
network as shown in Fig.~\ref{fig:1}, which contains 6 followers
and a virtual leader in a 3-dimensional space. The graph $G$ associated with the communication network is unweighted, i.e., $ a_{ij}=a_{ji}=1 $ if there exists an information exchange between the agents $ i $ and $ j $, and $ b_{i}=1 $ if the agent $ i $ can obtain the information from the leader.

\begin{figure*}[htbp]
\centering
\begin{minipage}[t]{0.3 \linewidth}
\centering
\includegraphics[width=0.93\textwidth]{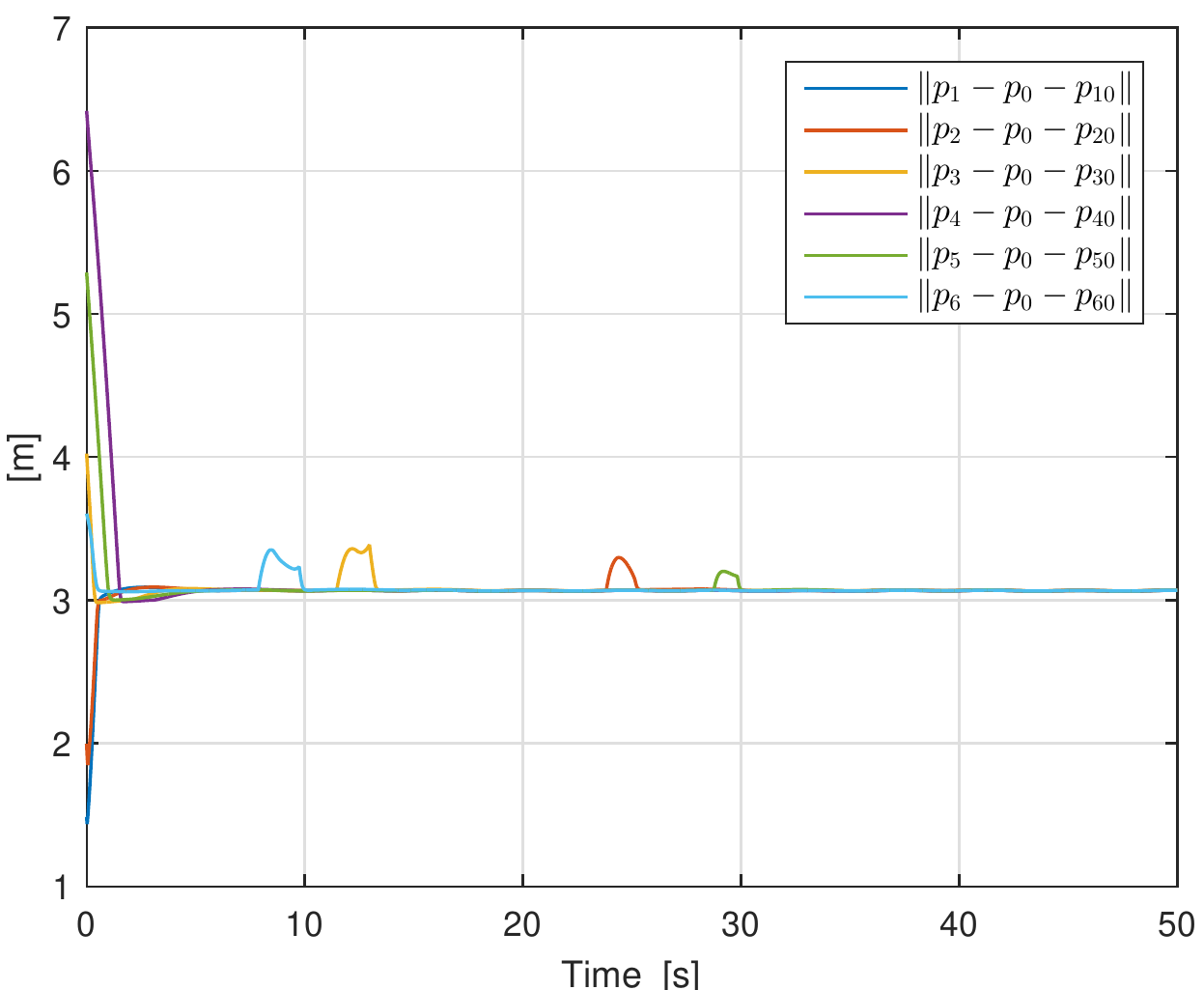}
\caption{The cooperative tracking behavior of robots.}
\label{fig:6}
\end{minipage}%
\hfill
\begin{minipage}[t]{0.3 \linewidth}
\centering
\includegraphics[width=1\textwidth]{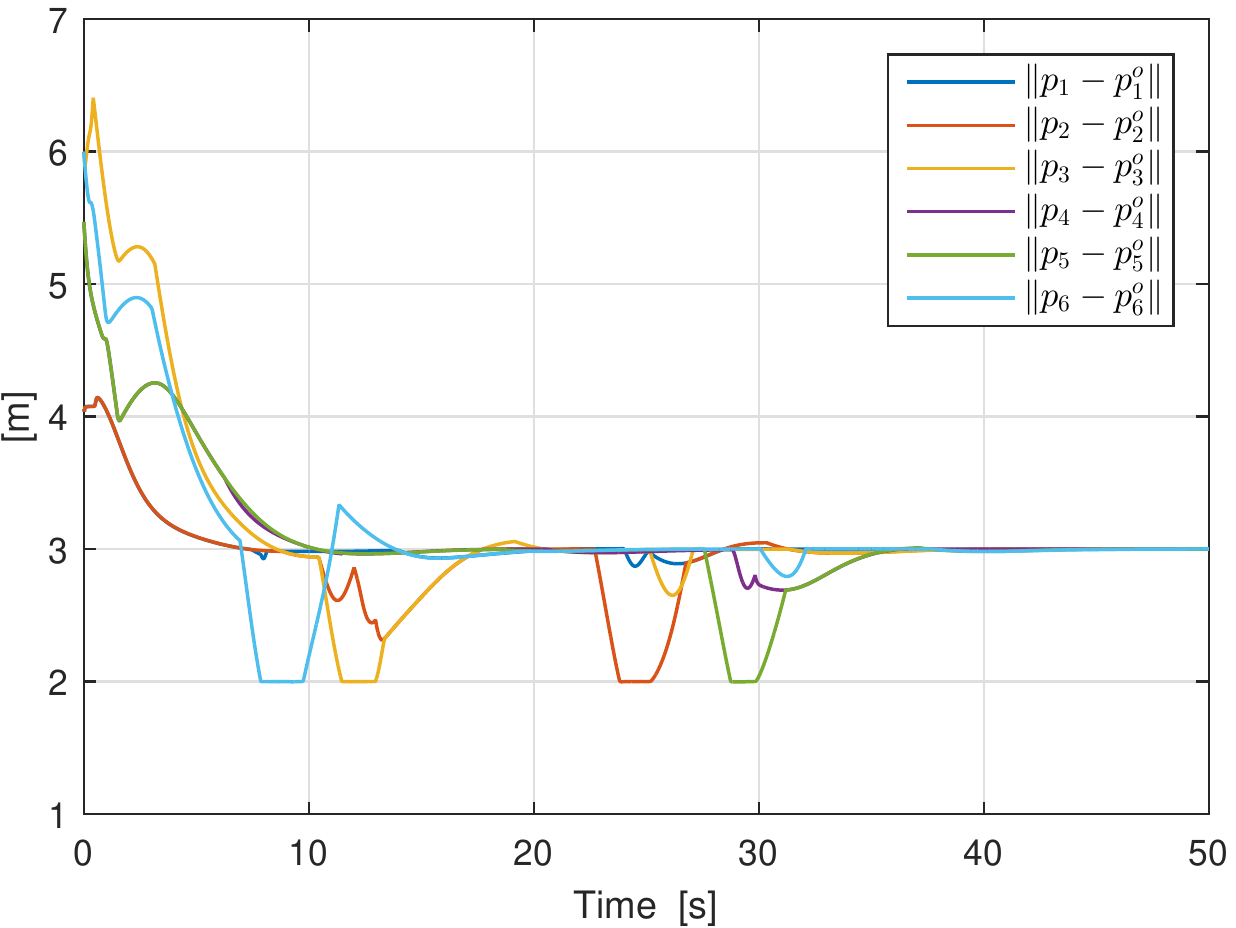}
\caption{The distances between robots and obstacles.}
\label{fig:8}
\end{minipage}%
\hfill
\begin{minipage}[t]{0.35 \linewidth}
\centering
\includegraphics[width=0.95\textwidth]{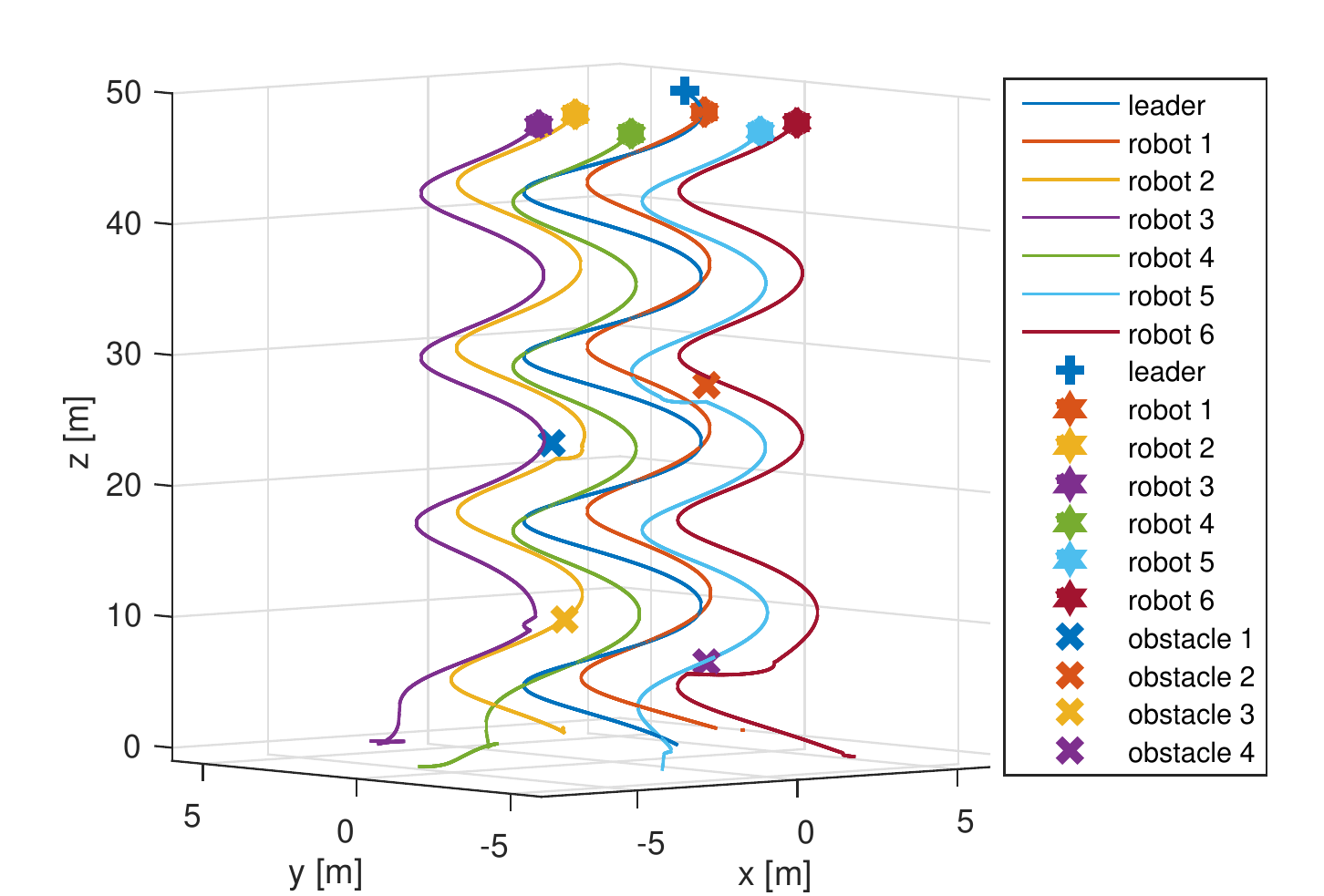}
\caption{Trajectories of the six robots in the environment with four obstacles.}
\label{fig:9}
\end{minipage}
\end{figure*}

The trajectory of leader is $ p_{0}= [2\cos(0.05t),2\sin(0.05t),\\0.1t]^{\top}$. The initial positions of the robots are
$p_{1}=[6,2,0]^{\top}$m, $p_{2}=[3,3+\sqrt{3},0]^{\top}$m, $ p_{3}=[-3,3+\sqrt{3},0]^{\top}$m, $p_{4}=[-7,-1,0]^{\top}$m, $ p_{5}=[-3,-3-\sqrt{3},0]^{\top}$m, $ p_{6}=[3,-3-\sqrt{3},0]^{\top}$m.
The positions of environmental obstacles are $O_{1}=[0,2,23]^{\top}$m, $ O_{2}=[1,-2,28]^{\top}$m, $O_{3}=[-1.5,0,10]^{\top}$m, $ O_{4}=[1,-2,7]^{\top}$m. The radius of its repulsive zone in \eqref{2} is $d=2$m. The  time-varying desired relative positions between the leader and the followers are  $ p_{10}=[\frac{3\sqrt{3}}{2}+0.1\sin(0.1t),\frac{3}{2},0]^{\top} $m, $ p_{20}=[0.1\sin(0.1t),3,0]^{\top} $m,  $ p_{30}=[0.1\sin(0.1t)-\frac{3\sqrt{3}}{2},\frac{3}{2},0]^{\top} $m,  $ p_{40}=[0.1\sin(0.1t)-\frac{3\sqrt{3}}{2},-\frac{3}{2},0]^{\top} $m,  $ p_{50}=[0.1\sin(0.1t),-3,0]^{\top} $m,  $ p_{60}=[\frac{3\sqrt{3}}{2}+0.1\sin(0.1t),-\frac{3}{2},0]^{\top} $m. The design parameters are selected as $ \beta_{1}=\beta_{2}=0.6 $, $ r_{0}=0.9 $, $ r_{1}=1.2 $, $ r_{2}=0.6 $, $ K_{1}=0.4 $, $ K_{2}=0.6 $, $K_{3}=1$, $ r_{3}=6 $, $ r_{4}=5 $, $ r_{5}=3 $, $ r_{6}=5 $.
%, and $ \delta_{d}=0.1 $.

The task error of each agent is shown in Fig.~\ref{fig:6}, where the
curves describe the gap between the expected relative positions and the real relative positions. In the time intervals of 8s$\sim$10s, 11.5s$\sim$13.3s, 22s$\sim $25s, 28.7s$\sim $30s, the obstacle avoidance behavior of
the agents 6, 3, 2 and 5 take place, respectively, as shown in Fig.~\ref{fig:8}. The collision avoidance behavior has a higher priority in the desired velocity \eqref{v} and force the robots to move away from obstacles, resulting in a deviation from their desired trajectories for formation task.  Fig. \ref{fig:9} then shows the trajectories of the six robots in the formation. The simulation result shows that the proposed algorithm is effective in time-varying formation problem of a team of robots in an environment with obstacles.

\section{Conclusions}
\label{sec6}
This paper has proposed a novel fixed-time behavioral control method, which can be applied to distributed time-varying formation control of networked robots in an environment with obstacles. Using the null-space-based projection, the collision avoidance task and cooperative formation task are combined, leading to a desired driving velocity for each robot to achieve a time-varying formation in a fixed-time convergence while avoiding collisions and obstacles. The simulation result has shown the effectiveness of this method. The future work will further discuss the adverse effects from constraints of input.

\section*{ Appendix: Proof of Theorem~\ref{taskl4}.}
\renewcommand{\thesubsection}{\Alph{subsection}}

\label{Ap:task14}

\textit{Proof.} There are three cases to be discussed.

\textbf{Case A}: If $ |\tilde{\rho}_{io}(t)| > 0 $ and $ |\tilde{\rho}_{if}(t)| =0 $, then according to Lemma~\ref{l1}, there exists a settling time $T_{i,o}>0$ such that $ \|p_{i}-p_{i}^{o}\|>d $ and {$ \|p_{i}-\hat{p}_{i}-p_{i0}\|=0$},  $\forall~t\geq T_{i,o}$ .

\textbf{Case B}: If $ |\tilde{\rho}_{io}(t)| =0 $ and $ |\tilde{\rho}_{if}(t)| > 0 $, then according to Lemma~\ref{l1}, there exists a settling time $T_{i,f}>0$ such that $ \|p_{i}-p_{i}^{o}\|>d$ and {$ \|p_{i}-\hat{p}_{i}-p_{i0}\|=0$}, $\forall~t\geq T_{i,f}$.

\textbf{Case C}: If $ |\tilde{\rho}_{io}(t)| > 0$ and $ |\tilde{\rho}_{if}(t)| > 0 $, then $ \tilde{\rho}_{io}(t)< 0$ and $ \tilde{\rho}_{if}(t)< 0 $. The rest of the proof is focused on the analysis of this case. For each robot $i \in V$, { consider the two different subtask errors and define the Lyapunov function as follows:}
\begin{align}\label{taskl4_1}
V_{i,M}(\tilde{\rho}_{io},\tilde{\rho}_{if}):=\frac{1}{2}\gamma_{i,o}\tilde{\rho}_{io}^{2}+\frac{1}{2}\gamma_{f}\tilde{\rho}_{if}^{2},  
\end{align}
where $ \gamma_{i,o}$, $\gamma_{f} \in \mathbb{R}_{>0}$ satisfies  
$ \gamma_{i,o}\ge \gamma_{f} J_{if}(0)J^{\dagger}_{io}(0)+ \gamma_{i,\varepsilon}$ with  $\gamma_{i,\varepsilon}\in \mathbb{R}_{>0}$. Moreover, we define $L_0$ such that
\begin{equation} \label{eq:L0}
0<\max_{i\in V}\{|\tilde{\rho}_{io}(0)|,|\tilde{\rho}_{if}(0)|\}\le L_{0},
\end{equation}
where $\tilde{\rho}_{io}(0)$, $\tilde{\rho}_{if}(0) $ are the initial values of $ \tilde{\rho}_{io}(t) $ and $\tilde{\rho}_{if}(t) $, respectively. Taking the time derivative of $ V_{i,M} $ along the desired velocity \eqref{v} yields
{\small\begin{align}
\dot{V}_{i,M}
=&-\gamma_{i,o}\tilde{\rho}_{io}J_{io}[v_{io}+(I-J_{io}^{\dag}J_{io})v_{if}]{-\gamma_{i,o}\tilde{\rho}_{io}J_{i}^{o}v_{i}^{o}}\nm\\
&-\gamma_{f}\tilde{\rho}_{if}J_{if}[v_{io}+(I-J_{io}^{\dag}J_{io})v_{if}]
-\gamma_{f}\tilde{\rho}_{if}(J_{i\hat{f}}\dot{\hat{p}}_{i}+J_{i0}\dot{p}_{i0})
\nm\\
=&-\gamma_{i,o}\tilde{\rho}_{io}J_{io} J_{io}^{\dag} \lambda_{io}(\beta_{1}\tilde{\rho}_{io}^{[r_{1}]}+\beta_{2}\tilde{\rho}_{io}^{[r_{2}]})^{[r_{0}]}\nm \\
&-\gamma_{f} \tilde{\rho}_{if}J_{if} J_{io}^{\dag} [\lambda_{io}(\beta_{1}\tilde{\rho}_{io}^{[r_{1}]}+\beta_{2}\tilde{\rho}_{io}^{[r_{2}]})^{[r_{0}]}-J_{i}^{o}v_{i}^{o}]\nm
\\& +\gamma_{f} \tilde{\rho}_{if} \tilde{J} [\lambda_{f}(\beta_{1}\tilde{\rho}_{if}^{[r_{1}]}+\beta_{2}\tilde{\rho}_{if}^{[r_{2}]})^{[r_{0}]}-J_{i\hat{f}}\dot{\hat{p}}_{i}-J_{i0}\dot{p}_{i0}]\nm
\\ \label{16}
& -\gamma_{f}\tilde{\rho}_{if}(J_{i\hat{f}}\dot{\hat{p}}_{i}+J_{i0}\dot{p}_{i0}),
\end{align}}
where $\tilde{J}: = J_{if}(I-J_{io}^{\dag}J_{io})J^{\dagger}_{if}$, and
$ J_{io}(I-J_{io}^{\dag}J_{io})=0 $. We further scale \eqref{16}  as
{\small\begin{align}\label{18}
\dot{V}_{i,M}\le&-\gamma_{i,o}\lambda_{io}\tilde{\rho}_{io}(\beta_{1}\tilde{\rho}_{io}^{[r_{1}]}+\beta_{2}\tilde{\rho}_{io}^{[r_{2}]})^{[r_{0}]}\nm
\\&+\gamma_{f}\lambda_{io}|J_{if}J^{\dagger}_{io}||\tilde{\rho}_{if}||(\beta_{1}\tilde{\rho}_{io}^{[r_{1}]}+\beta_{2}\tilde{\rho}_{io}^{[r_{2}]})^{[r_{0}]}|\nm
\\&-\gamma_{f}\lambda_{f}\tilde{\rho}_{if} \tilde{J} (\beta_{1}\tilde{\rho}_{if}^{[r_{1}]}+\beta_{2}\tilde{\rho}_{if}^{[r_{2}]})^{[r_{0}]} +\gamma_{f}|\tilde{\rho}_{if}|\|J_{if}\|\|v_{i}^{o}\|\nm
\\
&+\gamma_{f}|\tilde{\rho}_{if}|\|J_{i\hat{f}}\|\|\dot{\hat{p}}_{i}\| 
+\gamma_{f}|\tilde{\rho}_{if}|\|J_{i0}\|\|\dot{p}_{i0}\|,
\end{align}}
From this point, the proof goes into the following two directions.

\textbf{(a)} $ |\tilde{\rho}_{io}|\ge |\tilde{\rho}_{if}| > 0$

In this case, we have $|(\beta_{1}\tilde{\rho}_{io}^{[r_{1}]}+\beta_{2}\tilde{\rho}_{io}^{[r_{2}]})^{[r_{0}]}|\ge |(\beta_{1}\tilde{\rho}_{if}^{[r_{1}]}+\beta_{2}\tilde{\rho}_{if}^{[r_{2}]})^{[r_{0}]}|> 0$. First, we prove that $ \tilde{\rho}_{io}(t) $ is bounded if initial value $ \tilde{\rho}_{if}(0)$ is bounded. From the relations $ \gamma_{i,o}\ge\gamma_{f} J_{if}(0)J^{\dagger}_{io}(0)+ \gamma_{i,\varepsilon}$ and $ \frac{|\tilde{\rho}_{io}(t)|}{|\tilde{\rho}_{if}(t)|}\ge 1$, we obtain from \eqref{18} that 
{\small \begin{align}\label{d22}
&\dot{V}_{i,M}(0)
\le-\gamma_{i,o}\lambda_{io}\tilde{\rho}_{io}(0)(\beta_{1}\tilde{\rho}_{io}^{[r_{1}]}(0)+\beta_{2}\tilde{\rho}_{io}^{[r_{2}]}(0))^{[r_{0}]}\nm\\&+\gamma_{f}\lambda_{io}|J_{if}(0)J^{\dagger}_{io}(0)||\tilde{\rho}_{if}(0)||(\beta_{1}\tilde{\rho}_{io}^{[r_{1}]}(0)+\beta_{2}\tilde{\rho}_{io}^{[r_{2}]}(0))^{[r_{0}]}|\nm\\&-\gamma_{f}\lambda_{f}\tilde{\rho}_{if}(0)J_{if}(0)(I-J_{io}^{\dag}(0)J_{io}(0))J^{\dagger}_{if}(0)\nm\\
&\times(\beta_{1}\tilde{\rho}_{if}^{[r_{1}]}(0)+\beta_{2}\tilde{\rho}_{if}^{[r_{2}]}(0))^{[r_{0}]}+\gamma_{f}|\tilde{\rho}_{if}(0)|\|J_{if}(0)\|\|v_{i}^{o}(0)\|\nm\\&+\gamma_{f}|\tilde{\rho}_{if}(0)|\|J_{i\hat{f}}(0)\|\|\dot{\hat{p}}_{i}(0)\|+\gamma_{f}|\tilde{\rho}_{if}(0)|\|J_{i0}(0)\|\|\dot{p}_{i0}(0)\|\nm\\
&\le-\gamma_{i,o}\lambda_{iv}|\tilde{\rho}_{io}(0)||(\beta_{1}\tilde{\rho}_{io}^{[r_{1}]}(0)+\beta_{2}\tilde{\rho}_{io}^{[r_{2}]}(0))^{[r_{0}]}|\le 0,
\end{align}}
\if0 where $ |J_{io}J^{\dagger}_{if}|\le L_{iof}$, $ |J_{if}J^{\dagger}_{io}| \le L_{ifo}$, $ \|J_{io}\|\le \sqrt{d^{2}+2L_{0}} $ and $ \|J_{i\hat{f}}\|\le \sqrt{d_{i0}^{2}+2L_{0}} $ are utilized according to Eqs. \eqref{jio} and \eqref{jif}, \fi
where $ \lambda_{i,o} $ is state-dependent and designed to satisfy 
{\small \begin{align} \label{eq:constrain}
\lambda_{i,o}\ge& {\frac{\gamma_{f}\|J_{if}(0)\|\sup_{t\ge 0}(\|v_{i}^{o}\|+\|\dot{\hat{p}}_{i}\|+\|\dot{p}_{i0}\|)}{\gamma_{i,\varepsilon}|(\beta_{1}\tilde{\rho}_{if}^{[r_{1}]}(0)+\beta_{2}\tilde{\rho}_{if}^{[r_{2}]}(0))^{[r_{0}]}|}}+\frac{\gamma_{i,o}\lambda_{iv}}{\gamma_{i,\varepsilon}},
\end{align} }
with $ \lambda_{iv}>0 $ an auxiliary design parameter.
The constraint \eqref{eq:constrain} is used to guarantee $ \dot{V}_{i,M}(0)\le 0 $. It follows from \eqref{eq:L0} and \eqref{d22} that  
$ |\tilde{\rho}_{io}(\Delta t)|\le |\tilde{\rho}_{io}(0)|\le L_{0}$, for any $ \Delta t >0$. We can further show that $ \dot{V}_{i,M}( \Delta t )\le 0 $ if \eqref{eq:constrain} holds. Therefore, we conclude that if \eqref{eq:constrain} is satisfied, then $ \dot{V}_{i,M}(t)\le 0 $ for any finite $ t $, which implies that $ |\tilde{\rho}_{if}(t)|\le |\tilde{\rho}_{if}(0)|\le L_{0} $. 

Next, with the relations $-|\tilde{\rho}_{io}|\le -|\tilde{\rho}_{if}| $ and $-|(\beta_{1}\tilde{\rho}_{io}^{[r_{1}]}+\beta_{2}\tilde{\rho}_{io}^{[r_{2}]})^{[r_{0}]}|\le- |(\beta_{1}\tilde{\rho}_{if}^{[r_{1}]}+\beta_{2}\tilde{\rho}_{if}^{[r_{2}]})^{[r_{0}]}| $, we rewrite \eqref{d22} as
{\small \begin{align}\label{22}
\dot{V}_{i,M}\le&-\frac{\gamma_{i,o}\lambda_{iv}}{2}|\tilde{\rho}_{io}||(\beta_{1}\tilde{\rho}_{io}^{[r_{1}]}+\beta_{2}\tilde{\rho}_{io}^{[r_{2}]})^{[r_{0}]}|\nm\\&-\frac{\gamma_{i,o}\lambda_{iv}}{2}|\tilde{\rho}_{if}||(\beta_{1}\tilde{\rho}_{if}^{[r_{1}]}+\beta_{2}\tilde{\rho}_{if}^{[r_{2}]})^{[r_{0}]}|,\nm\\ \le& -\eta_{iM1}{V}_{i,M}^{\frac{r_{1}r_{0}+1}{2}}-\eta_{iM2}{V}_{i,M}^{\frac{r_{2}r_{0}+1}{2}},
\end{align}}
where 
{\small 
\begin{align*}
\eta_{iM1}:=\min\left\{\frac{\gamma_{i,o}\lambda_{iv}\beta_{1}}{2^{\frac{3-r_{1}r_{0}}{2}}}\left(\frac{2}{\gamma_{i,o}}\right)^{\frac{2}{r_{1}r_{0}+1}}, \frac{\gamma_{i,o}\lambda_{iv}\beta_{1}}{2^{\frac{3-r_{1}r_{0}}{2}}}\left(\frac{2}{\gamma_{f}}\right)^{\frac{2}{r_{1}r_{0}+1}} \right\}, \\
\eta_{iM2}:=\min\left\{\frac{\gamma_{i,o}\lambda_{iv}\beta_{2}}{2}\left(\frac{2}{\gamma_{i,o}}\right)^{\frac{2}{r_{2}r_{0}+1}}, \frac{\gamma_{i,o}\lambda_{iv}\beta_{2}}{2}\left(\frac{2}{\gamma_{f}}\right)^{\frac{2}{r_{2}r_{0}+1}} \right\}.
\end{align*}}%
It then obtain from Lemma \ref{ld1} that for any initial values $(\tilde{\rho}_{io}(0),\tilde{\rho}_{if}(0))$, there exists a settling time 
\begin{equation} \label{eq:Ti1}
T_{i,1}=\frac{2}{\eta_{iM1}(r_{1}r_{0}-1)}+\frac{2}{\eta_{iM2}(1-r_{2}r_{0})}
\end{equation} 
such that $ \|p_{i}-p_{i}^{o}\|\ge 0 $ and $ \|p_{i}-\hat{p}_{i}-p_{i0}\|=0 $, $\forall~t\geq T_{i,1}$. 
\if0 
From Eq. \eqref{22}, we further get
\begin{align}\label{23}
\dot{V}_{M}\le& -\eta_{M1}{V}_{M}^{\frac{r_{1}r_{0}}{2}}-\eta_{M2}{V}_{M}^{\frac{r_{2}r_{0}}{2}},
\end{align}
where $ \eta_{M1}:= \min_{i}\{n^{\frac{1-r_{1}r_{0}}{2}}\eta_{iM1}\}$, $ \eta_{M2}:=\min_{i}\{\eta_{iM2}\} $. Then from Eq. \eqref{23} and Lemma \ref{ld1}, for any $\phi_{s}>0$, $(\tilde{\rho}_{o}(0),\tilde{\rho}_{f}(0))\in\bR^{n}\times\bR^{n}$, there exists a settling time $T_{1}=\frac{2}{\eta_{M1}(r_{1}r_{0}-1)}+\frac{2}{\eta_{M2}(1-r_{2}r_{0})}$ such that $ \|p_{i}-p_{i}^{o}\|\ge \sqrt{d^{2}-2\phi_{s}} $ and $ \|\tilde{\rho}_{if}(t)\| \leq \phi_{s} $ for all $i=1,\dots,n$ and $t\geq T_{1}$.
\fi 

{\textbf{(b)}} $ 0<|\tilde{\rho}_{io}|< |\tilde{\rho}_{if}|\le L_{0} $

Now, we have $0<|(\beta_{1}\tilde{\rho}_{io}^{[r_{1}]}+\beta_{2}\tilde{\rho}_{io}^{[r_{2}]})^{[r_{0}]}|< |(\beta_{1}\tilde{\rho}_{if}^{[r_{1}]}+\beta_{2}\tilde{\rho}_{if}^{[r_{2}]})^{[r_{0}]}|\le L_{0}^{*}$, where $ L_{0}^{*}:= (\beta_{1}L_{0}^{r_{1}}+\beta_{2}L_{0}^{r_{2}})^{r_{0}} $. We first show the boundedness of $ \tilde{\rho}_{if}(t) $ for a bounded initial value $ \tilde{\rho}_{if}(0)$. It follows from $ \gamma_{i,o}\ge\gamma_{f} J_{if}(0)J^{\dagger}_{io}(0)\sup_{t\ge 0}\frac{|\tilde{\rho}_{if}(t)|}{|\tilde{\rho}_{io}(t)|}+ \gamma_{i,\varepsilon}$ and $ \frac{|\tilde{\rho}_{if}(t)|}{|\tilde{\rho}_{io}(t)|}\ge1$ that
{\small \begin{align}\label{d23}
&\dot{V}_{i,M}(0)\le-\gamma_{i,o}\lambda_{io}\tilde{\rho}_{io}(0)(\beta_{1}\tilde{\rho}_{io}^{[r_{1}]}(0)+\beta_{2}\tilde{\rho}_{io}^{[r_{2}]}(0))^{[r_{0}]}\nm\\&+\gamma_{f}\lambda_{io}|J_{if}(0)J^{\dagger}_{io}(0)||\tilde{\rho}_{if}(0)||(\beta_{1}\tilde{\rho}_{io}^{[r_{1}]}(0)+\beta_{2}\tilde{\rho}_{io}^{[r_{2}]}(0))^{[r_{0}]}|\nm\\&-\gamma_{f}\lambda_{f}\tilde{\rho}_{if}(0)J_{if}(0)(I-J_{io}^{\dag}(0)J_{io}(0))J^{\dagger}_{if}(0)\nm\\
&\times(\beta_{1}\tilde{\rho}_{if}^{[r_{1}]}(0)+\beta_{2}\tilde{\rho}_{if}^{[r_{2}]}(0))^{[r_{0}]}+\gamma_{f}|\tilde{\rho}_{if}(0)|\|J_{if}(0)\|\|v_{i}^{o}(0)\|\nm\\&+\gamma_{f}|\tilde{\rho}_{if}(0)|\|J_{i\hat{f}}(0)\|\|\dot{\hat{p}}_{i}(0)\|+\gamma_{f}|\tilde{\rho}_{if}(0)|\|J_{i0}(0)\|\|\dot{p}_{i0}(0)\|\nm\\ &\le-\gamma_{i,\varepsilon}\lambda_{i,o}|\tilde{\rho}_{io}(0)||(\beta_{1}\tilde{\rho}_{io}^{[r_{1}]}(0)+\beta_{2}\tilde{\rho}_{io}^{[r_{2}]}(0))^{[r_{0}]}|\nm\\&-\gamma_{f}\lambda_{iv}|\tilde{\rho}_{if}(0)||(\beta_{1}\tilde{\rho}_{if}^{[r_{1}]}(0)+\beta_{2}\tilde{\rho}_{if}^{[r_{2}]}(0))^{[r_{0}]}|\le 0,
\end{align} }
where $ \lambda_{f} $ is state-dependent and designed to satisfy 
{\small\begin{align}\label{lambdaio2}
\lambda_{f}\ge &{\gamma_{i,f}\lambda_{iv}} + \\
& {\frac{\gamma_{f}\|J_{if}(0)\|\sup_{t\ge 0}(\|v_{i}^{o}\|+\|\dot{\hat{p}}_{i}\|+\|\dot{p}_{i0}\|)}{J_{if}(0)(I-J_{io}^{\dag}(0)J_{io}(0))J^{\dagger}_{if}(0)|(\beta_{1}\tilde{\rho}_{if}^{[r_{1}]}(0)+\beta_{2}\tilde{\rho}_{if}^{[r_{2}]}(0))^{[r_{0}]}|}}.\nm
\end{align}}%
With \eqref{lambdaio2}, $ \dot{V}_{i,M}(t)\le 0 $ for any finite $ t $, and thus $ |\tilde{\rho}_{if}(t)|\le |\tilde{\rho}_{if}(0)|\le L_{0} $ is guaranteed.  Furthermore, according to  \eqref{d23}, we can be rewrite \eqref{18} as  
{\small \begin{align}\label{d25}
\dot{V}_{i,M}\le&-\gamma_{i,\varepsilon}\lambda_{i,o}|\tilde{\rho}_{io}(0)||(\beta_{1}\tilde{\rho}_{io}^{[r_{1}]}(0)+\beta_{2}\tilde{\rho}_{io}^{[r_{2}]}(0))^{[r_{0}]}|\nm\\ &-\gamma_{f}\lambda_{iv}|\tilde{\rho}_{if}(0)||(\beta_{1}\tilde{\rho}_{if}^{[r_{1}]}(0)+\beta_{2}\tilde{\rho}_{if}^{[r_{2}]}(0))^{[r_{0}]}|,\nm
\\ 
\le& -\eta_{iM3}{V}_{i,M}^{\frac{r_{1}r_{0}+1}{2}}-\eta_{iM4}{V}_{i,M}^{\frac{r_{2}r_{0}+1}{2}},
\end{align}}
where 
{\small
\begin{align*}
\eta_{iM3}:= \min \left\{\frac{\gamma_{i,\varepsilon}\lambda_{iv}\beta_{1}}{2^{\frac{3-r_{1}r_{0}}{2}}}\left(\frac{2}{\gamma_{i,o}}\right)^{\frac{2}{r_{1}r_{0}+1}}, \frac{\gamma_{f}\lambda_{iv}\beta_{1}}{2^{\frac{3-r_{1}r_{0}}{2}}}\left(\frac{2}{\gamma_{f}}\right)^{\frac{2}{r_{1}r_{0}+1}} \right\}, 
\\
\eta_{iM4}:= \min \left\{\frac{\gamma_{i,\varepsilon}\lambda_{iv}\beta_{2}}{2}\left(\frac{2}{\gamma_{i,o}}\right)^{\frac{2}{r_{2}r_{0}+1}}, \frac{\gamma_{f}\lambda_{iv}\beta_{2}}{2}\left(\frac{2}{\gamma_{f}}\right)^{\frac{2}{r_{2}r_{0}+1}} \right\}.
\end{align*}
}%
Then from Lemma \ref{ld1}, we obtain  $ \|p_{i}-p_{i}^{o}\|\ge d $ and $ \|p_{i}-\hat{p}_{i}-p_{i0}\|=0 $ for all  
\begin{equation} \label{eq:Ti2}
t \geq { T_{i,2}}:=\frac{2}{\eta_{iM3}(r_{1}r_{0}-1)}+\frac{2}{\eta_{iM4}(1-r_{2}r_{0})}.
\end{equation}

From the above discussions, we conclude that if the merged driving velocity in \eqref{v} is applied to each robot $i \in V$, then for any initial values $(\tilde{\rho}_{io}(0),\tilde{\rho}_{if}(0))$,  there exists a settling time ${ T_{i}}:=\max\{T_{i,o},T_{i,f},T_{i,1},{ T_{i,2}}\}$
with $T_{i,o},T_{i,f}$ defined in Lemma~\ref{l1} and Lemma~\ref{l4}, respectively, such that $ \|p_{i}-p_{i}^{o}\|\ge d $ and $ \|p_{i}-\hat{p}_{i}-p_{i0}\|=0$, $ \forall~t\geq T_{i}$.

\bibliographystyle{IEEEtran}
\bibliography{refv}
\end{document}